\documentclass[11pt,oneside]{amsart}
\usepackage[height=9in,width=6in,marginparwidth=0.8in]{geometry}
\usepackage{hyperref}
\hypersetup{
	pdfauthor   = {},
	pdftitle    = {Rigorous integration of algebraic functions},
	pdfsubject  = {},
	pdfkeywords = {},
	colorlinks=true,
	breaklinks=true, urlcolor=blue, linkcolor=blue, citecolor=blue, bookmarksopen=true}
\usepackage{amsmath}
\usepackage{amssymb}
\usepackage{amsthm}
\usepackage[utf8]{inputenc}
\usepackage{tikz-cd} 
\usepackage[foot]{amsaddr}
\usepackage{graphicx}
\usepackage{caption}
\usepackage{subcaption}

\newtheorem{theorem}{Theorem}[section]
\newtheorem{corollary}[theorem]{Corollary}
\newtheorem{lemma}[theorem]{Lemma}
\newtheorem{prop}[theorem]{Proposition}
\theoremstyle{definition}

\newtheorem{remark}[theorem]{Remark}
\newtheorem{example}[theorem]{Example}
\newtheorem{strategy}{Strategy}

\numberwithin{equation}{section}
		
\newcommand{\abs}[1]{\left\lvert #1 \right\rvert}

\newcommand{\pround}[1]{\left( #1 \right)}
\newcommand{\psquare}[1]{\left[ #1 \right]}
\newcommand{\pbrace}[1]{\left\{ #1 \right\} }

\DeclareMathOperator{\acosh}{acosh}

\newcommand{\CC}{\mathbb{C}}

\newcommand{\Etol}{E_\mathrm{tol}}
\newcommand{\Mproxy}{M_\mathrm{proxy}}

\title{Rigorous numerical integration of algebraic functions}
\date{August 23, 2022}

\author{Nils Bruin}
\address[1, 3]{Department of Mathematics, Simon Fraser University, Burnaby, BC, V5A 1S6, Canada}
\email{nbruin@sfu.ca}
\thanks{Nils Bruin acknowledges the support of the Natural Sciences and Engineering Research Council of Canada (NSERC), funding reference number RGPIN-2018-04191.}

\author{Linden Disney-Hogg}
\address[2]{School of Mathematics, University of Edinburgh, 
James Clerk Maxwell Building, Peter Guthrie Tait Rd, Edinburgh, EH9 3FD, UK}
\email{a.l.disney-hogg@sms.ed.ac.uk}
\thanks{The research of Linden Disney-Hogg is supported by a UK Engineering and Physical Sciences Research Council (EPSRC) studentship.}

\author{Wuqian Effie Gao}


\begin{document}
\begin{abstract}
We present an algorithm which uses Fujiwara's inequality to bound algebraic functions over ellipses of a certain type, allowing us to concretely implement a rigorous Gauss-Legendre integration method for algebraic functions over a line segment. We consider path splitting strategies to improve convergence of the method and show that these yield significant practical and asymptotic benefits.
We implemented these methods to compute period matrices of algebraic Riemann surfaces and these are available in SageMath.
\end{abstract}

\maketitle

\section{Introduction}

We consider the problem of numerically approximating a line integral of an algebraic integrand over $\CC$, with the particular application in mind of computing period matrices of algebraic curves. We are particularly interested in doing so with arbitrary precision.

We consider integrals along straight line segments, so with a suitable choice of coordinates the integral can be expressed as
\[\int_{-1}^1 g(z)dz,\]
where $g$ satisfies a polynomial algebraic equation
\begin{equation}\label{E:defining_f}
0= f(z,g)=a_0(z)g(z)^n+\cdots+a_n(z),\text{ where }a_i(z)\in\CC[z].
\end{equation}
We furthermore assume that $a_0(z)\neq 0$ for $z\in [-1,1]$ and that the discriminant of $f$ relative to $g$, denoted by $\operatorname{disc}_g(f)\neq 0$, for $z\in [-1,1]$. This implies that all $n$ functions satisfying \eqref{E:defining_f} are regular for $z\in [-1,1]$ and take distinct values for each $z\in[-1,1]$.

The integrand is fully determined by the equation $f(z,g)$ together with an (approximate) value $g(-1)$, as certified homotopy continuation methods \cite{Kranich2016} can be used to obtain values $g(z)$ for $z\in [-1,1]$.

These conditions are naturally met for the computation of periods along the lines of \cite{Bruin2019}. For applications to computation of Abel-Jacobi maps, it may be harder to ensure these restrictions on paths are met and additional techniques may be required.

Since we know our integrand is holomorphic along the integration path and \emph{a priori} have no other information, we use the Gauss-Legendre integration scheme. We use a fairly standard error bound (see Theorem~\ref{T: N required}) in terms of the variation of the integrand on an elliptical region with the points $z=\pm1$ as focal points.

We can obtain a very rough bound on the values of our integrand from just the algebraic equation $f(z,g)$ through Fujiwara's bound on the absolute values of roots of complex polynomials (see Corollary~\ref{C:fujiwara_uniform_bound}). 
We use a Taylor series expansion of the integrand in combination with Cauchy's form of the error term to obtain a reasonable bound on the integrand on circular regions well inside the radius of convergence, which is the distance to the closest critical point of the integrand (see Corollary~\ref{C:taylor_uniform_bound}).

One approach would then be to determine an ellipse with focal points $z=-1,1$ that does not contain any of the critical points and bound the integrand by covering the elliptical region with appropriately chosen discs. That could force us to use a very eccentric elliptical region, and unfortunately our error bound for Gauss-Legendre behaves rather badly for such ellipses.

Instead, we propose a strategy that adaptively bisects the integration path (see Section~\ref{S:path_splitting_strategy}, Strategy~\ref{strat:main}). While that leaves us with computing multiple Gauss-Legendre quadratures, each of them fits inside a much less eccentric ellipse. As we show in Section~\ref{subsec: circle choosing}, this approach performs much better when a critical point lies close to the integration path. We illustrate this with some numerical experiments (see Figure~\ref{fig:proxy_heatmaps}) as well as some asymptotic results, such as the following (see Proposition~\ref{P:N1_asymptotics} for a more precise statement and proof).

\begin{prop}
Take $g(z)$ with a critical point at $z=0$ such that $\lim_{z\to 0} g(z)z^v$ exists and suppose that the function is bounded on the annuli $\delta \leq |z|<2$ for $\delta>1$. Choose an absolute error tolerance $\Etol>0$.

Then using Strategy~\ref{strat:main}, we can compute an approximation $I_q$ of the path integral $\int_{-1+iq}^{1+iq} g(z) dz$ such that $|I_q-\int_{-1+iq}^{1+iq} g(z) dz| < \Etol$ with a number of function evaluations $N_q$ where, for $q\to 0$, we have
	\[N_q\sim \begin{cases}
	\log(q^{-1})^2 & \text{ for }\;\; v\geq 1 \, , \\
	(\log(\log q^{-1}))^2 & \text{ for }\;\; 0 < v < 1 \, . \\		
\end{cases}\]
\end{prop}

The situation with $v<1$ occurs commonly when integrating regular differentials on algebraic Riemann surfaces. Note, however, that the asymptotics suggested in the proposition are rarely realised in practice: usually, paths will be separated from critical points by at least $\Etol$. It is more meant to reassure that path splitting, which is the characterising feature of Strategy~\ref{strat:main}, allows $\Etol$ to be the dominant factor in complexity even in the presence of critical points.

The main reference for us is a heuristic method described in \cite{Borwein2004}, based on iteratively doubling the number of evaluation nodes and estimating the accuracy of the approximation by tracking the differences between successive approximations. This method works quite well in practice, but in its standard formulation does not split the path. As we illustrate in Proposition~\ref{P:N2_asymptotics}, in those cases one expects the number of evaluation points to be more on the order of $q^{-1}$. This behaviour bears out in actual applications; also for the heuristic method.

A situation where one may be forced to compute line integrals near critical points occurs when an integrand has two such points close together and the path needs to separate the two. In Section~\ref{S:close_pair_of_critical_points}, we look at a realistic example with $v=\frac{1}{2}$. There we see that moving two critical points close forces coefficient growth as well, which handily overtakes the $\log(\log q^{-1})^2$ behaviour. We are still getting considerably better results than without path splitting.

Our implementation is available in SageMath~9.6 (see \cite{SageMath}), where the desired integration method used by \texttt{RiemannSurface} can be selected by \texttt{integration\_method = 'rigorous'} (the method described here, the default) or \texttt{integration\_method = 'heuristic'} (the old heuristic method based on \cite{Borwein2004}).

There is a general-purpose implementation for rigorously computing integrals as part of \cite{Johansson2013} which uses similar path splitting strategies (see \cite{Petras2002,Johansson2018}). For our application, the additional information we can obtain on our integrand was difficult to communicate through the provided interface, so we chose to build a separate implementation.

We also point out that \cite{MolinNeurohr2019} considers similar methods for the specific case of superelliptic curves. Interestingly, they find that the special shape of their integrands means integration schemes other than Gauss-Legendre actually provide better performance, and allow for more accurate error-term estimates. If faced with curves that allow for a model of that form, their specific methods lead to much better performance. The advantage of the approach we describe comes from the fact that it applies to arbitrary plane algebraic curves.

Finally, it is also worth pointing out that deformation methods as described in \cite{Sertoz2019} give an entirely different complexity picture and are also suitable for obtaining results with certified accuracy.

\section{Error bounds on numerical integration}

\subsection{Gauss-Legendre integration}

We recall that a \emph{linear quadrature scheme of order $N$ for integrals over $[-1,1]$} is a collection of $N$ nodes $\{x_1,\ldots,x_N\}\subset [-1,1]$ together with weights $w_1,\ldots,w_N$. It defines an approximation of an integral $I=\int_{-1}^1 g(x)dx$ by
\[I_N(g)=\sum_{i=1}^N w_ig(x_i), \quad\text{ with error }\quad E_N(g)=\abs{I-I_N(g)} \, . \]
The Gauss-Legendre quadrature is the unique one of order $N$ such that $E_N(g)=0$ for all polynomials $g$ of degree at most $2N-1$. Writing $P_N(x)$ for the degree-$N$ Legendre polynomial and $x_1,\ldots,x_N$ for its roots, the $N$-th order Gauss-Legendre quadrature scheme is obtained by taking $\{x_1,\ldots,x_N\}$ as the nodes and taking the weights to be
\[
w_i = \frac{2}{(1-x_i^2) P_N^\prime(x_i)^2} \, . 
\]
In order to bound $E_N(g)$, we consider elliptical regions with foci $z_1,z_2$, defined for $r>0$ by
\[
L(z_1,z_2,r)=\pbrace{z \in \mathbb{C} \mid \abs{z-z_1} + \abs{z-z_2} \leq \cosh(r)|z_2-z_1|}.
\]
The following theorem follows from approximating the integrand with Chebychev polynomials.
\begin{theorem}[\cite{Rabinowitz1969, NeurohrThesis}]\label{T: N required}
Suppose that $g$ is holomorphic on $L(-1,1,r)$ and that $\abs{g(z)}\leq M$ for $z\in L(-1,1,r)$. Then we have
\[
E_N(g)=\abs{\int_{-1}^1g(z)dz-I_N(g)}\leq \left( \pi + \frac{64}{15(e^{2r}-1)} \right ) \frac{M}{e^{2Nr}}\,.
\]
\end{theorem}

With a given tolerance $\Etol$ and calculation of the bound $M$ we can solve for $N$. Since we can transform any integral along a line segment to an integral along $[-1,1]$ we obtain the following.

\begin{corollary} Suppose $g$ is holomorphic and $\abs{g}$ is bounded by M on $L(z_1,z_2,r)$. Then with 
\begin{equation}\label{E:Nbound}
	N\geq \frac{1}{2r}\left[  \log\left(\pi+\frac{64}{15(e^{2r}-1)}\right) +\log M+\log\frac{1}{2}|z_2-z_1| -\log \Etol\right]
\end{equation}
and $h(z)=\frac{z_2-z_1}{2}g(\tfrac{1}{2}((1-z)z_1+(1+z)z_2))$ we have
\[\left|\int_{z_1}^{z_2}g(z)dz -  I_N(h)\right|\leq  \Etol \, .\]
\end{corollary}
\begin{proof}
We have $\int_{z_1}^{z_2} g(z) dz= \int_{-1}^{1}h(z)dz$.
Because $\abs{g}$ is bounded on $L(z_1, z_2, r)$ by $M$, we know that $\abs{h}$ is bounded on $L(-1, 1, r)$ by $\tfrac{1}{2}\abs{z_2-z_1}M$. The statement follows from applying Theorem~\ref{T: N required} and solving for $N$.
\end{proof}

\begin{remark}
Note that Gauss-Legendre quadrature is exact on constant functions, so $g(z)$ and $g(z)-g(z_0)$ will be integrated with the same error (at least in an exact field, see \S\ref{sec: rounding error} for comments about floating point arithmetic). Hence, it is sufficient to take $M$ to be a bound on the \emph{variation} of $g$ on $L(z_1, z_2, r)$.
\end{remark}

One can compute a sharp $M$ by finding the extremal values of $g(z)$ on the boundary of $L(z_1,z_2,r)$. In the next section we discuss how to compute a reasonable value for $M$ via a closed-form formula that can be relatively efficiently evaluated.

\subsection{Bounding the variation of an algebraic function}
\label{S:bound_variation}
We first show how to bound the variation of an algebraic function on a disk. To do so we use a standard bound obtained from Cauchy's form of the remainder term of a Taylor series for a holomorphic function $g$, adapting the approach of \cite{Kranich2016}.

We denote the closed disk with center $z_0$ and radius $\rho$ by
\[D(z_0,\rho)=\{z\in\CC : |z-z_0| \leq \rho\}\]
and its boundary, the circle with same center and radius, by $\partial D(z_0,\rho)$.

\begin{lemma}[\cite{Ahlfors1979}, pp.124-126]
	Let $g:U \to \mathbb{C}$ be a function holomorphic on an open subset $U\subset \CC$, and let the Taylor expansion around $z_0 \in U$ be given by 
	\[
	g(z) = g(z_0) + (z-z_0) g^\prime(z_0) + (z-z_0)^2 R(z).
	\]
	Suppose that $\rho>0$ is such that $g$ is holomorphic on the closed disk $D(z_0,\rho)$.
	With
	\[\tilde{M} = \max\big\{\abs{g(z)} : \abs{z-z_0}=\rho\big\} = \max \big\{\abs{g(z)} : z \in \partial D(z_0,\rho)\big\}\]
	we have
	\[|R(z)| \leq \frac{\tilde{M}}{\rho(\rho-|z-z_0|)} \text{ for all $z$ with $|z-z_0|<\rho$}.
	\]
\end{lemma}

\begin{corollary}\label{C:taylor_uniform_bound}
	For $0<\delta <\rho$, we have 
	\[
	\abs{g(z) - g(z_0)} < \delta \abs{g^\prime(z_0)} + \delta^2 \frac{\tilde{M}}{\rho(\rho-\delta)} \text{ for all $z$ with $|z-z_0|<\delta$}.
	\]
\end{corollary}

This leaves us with computing $\tilde{M}$. While theoretically this is exactly the same problem as computing $M$ in the first place, the dependence on $\rho$ and $\delta$ allows us to use a coarser method. Recall that for our applications, $g$ is an algebraic function and hence for every $z$, we have that $g(z)$ is a root of a complex polynomial $a_0(z)g(z)^n+\cdots+a_n(z)$.
\begin{lemma}[\cite{Fujiwara1916}]\label{lemma: Fujiwara}
Consider the complex polynomial $p$ given by
\[p(w) = \sum_{k=0}^n a_k w^{n-k}.\]
Then for any root $\alpha\in\CC$ with $p(\alpha)=0$ we have
	\[
	\abs{\alpha} < 2 \max \pbrace{ \abs{\frac{a_k}{a_0}}^{\frac{1}{k}} : k =1, \dots, n} \, .
	\]
\end{lemma}
Let us write
\[\begin{aligned}
a_0(z)&=a_{0,0}(z-\alpha_1)\cdots(z-\alpha_{d_0}),\\
a_i(z)&=a_{i,0}z^{d_i}+\cdots+a_{i,d_i}\text{ for }i=1,\ldots,n.\\
\end{aligned}
\]
Given $z_0$ and $\delta< \min\{ |z_0-\alpha_i|: i=1,\ldots,n\}$, we find a uniform lower bound for $|a_0(z)|$ on $z\in D(z_0,\delta)$ by
\[A_0=|a_{0,0}| \prod_{i=1}^n(|z_0-\alpha_i|-\delta)\]
and an upper bound on $|a_i(z)|$ via the triangle inequality, resulting in
\[A_i=|a_{i,0}|(|z_0|+\delta)^{d_i}+\cdots+|a_{i,d_i}|.\]
\begin{corollary}\label{C:fujiwara_uniform_bound}
With the notation above, we have for $z\in D(z_0,\delta)$ that
\[|g(z)|< \tilde{M}=2 \max\left\{\left(\frac{A_k}{A_0}\right)^{\frac{1}{k}}: k=1,\ldots,n\right\} \, . \]
\end{corollary}

\begin{remark}
    Alternative bounds on the absolute value of a root of a polynomial in terms of its coefficients exist, for example Lemma \ref{lemma: Fujiwara} is a special case of the result that, given $\lambda_k \in \mathbb{R}_{>0}$ such that $\sum_{k=1}^n \lambda_k^{-1} < 1$, $\abs{\alpha} < \max \pbrace{ \abs{\lambda_k \frac{a_k}{a_0}}^{\frac{1}{k}}}$. In this paper we do not make an exhaustive investigation of the best choice of bound, but we found in simple experiments that the form of Lemma \ref{lemma: Fujiwara} performed best, which is what one would expect if on average the values $\abs{ \frac{a_k}{a_0}}^{\frac{1}{k}}$ are of approximately the same size. This conclusion is obviously highly dependent on the sampling procedure for the algebraic functions. 
    Moreover, the strictest bound that could come from varying the $\lambda_k$ would be $\abs{\alpha} < \max \pbrace{ \abs{ \frac{a_k}{a_0}}^{\frac{1}{k}}}$, and so when we use $\log M$ to get a lower bound on $N$, we could at most reduce that lower bound by $\frac{\log 2}{2r}$. This is unlikely to give significant improvement to the method. 
\end{remark}


\section{Path splitting strategy}\label{S:path_splitting_strategy}

We have an algebraic function $g$ satisfying a polynomial equation \[f(z,g)=a_0(z)g(z)^n+\cdots+a_n(z)=0,\]
with $a_i\in \CC[z]$. Let $a_0(z)=a_{0,0}(z-\alpha_1)\cdots(z-\alpha_{d_0})$. We say the $\alpha_i$ are \emph{critical points} of $g$. Any pole of $g(z)$ would be a critical point. We assume that none of the $\alpha_i$ lie in the real interval $[-1,1]$. This property is ensured in our use case by the way the paths are chosen in the case of computing periods on algebraic Riemann surfaces as in \cite{Bruin2019}.

We want to approximate the integral $I=\int_{-1}^1 g(z)dz$ using Gauss-Legendre quadrature to within $\Etol$. With \eqref{E:Nbound} we can compute an appropriate order $N$ if we can bound the integrand on an elliptical region covering the interval $[-1, 1]$. Using the results in Section~\ref{S:bound_variation} we can bound the integrands on certain disks, as recorded in the lemma below.

\begin{lemma}
\label{L:bound_on_disk}
Let $D(z_j,\delta_j)$ be a disk in $\CC$ that contains none of the $\alpha_i$, that is $\delta_j < \rho_j =\min\{|z_j-\alpha_i|: i =1,\ldots, d_0\}$ and $\tilde{M}$ as in Corollary~\ref{C:fujiwara_uniform_bound}.  The bound $M=M(z_j,\delta_j)$ as given by Corollary~\ref{C:taylor_uniform_bound} yields a bound for $\abs{g(z)}$ for $z\in D(z_j,\delta_j)$.
\end{lemma}

Thus to complete a strategy for finding $N$ we need a way of choosing elliptical regions covering the line segment $[-1,1]$ and disks covering this region. We propose the following strategy. 

\begin{strategy}\label{strat:main}$ $
\begin{enumerate}
	\item Iteratively bisect $[-1,1]$ to a sequence $-1=z_0 < z_1 < \ldots < z_m=1$ such that
	\[\rho_j=\min_{i=1,\ldots,d_0}|\alpha_i-\tfrac{1}{2}(z_{j}+z_{j-1})| > \delta_{\min,j}=\tfrac{1}{2}|z_j-z_{j-1}|\text{ for }j=1,\ldots,m\,.\]
	\item For $j=1,\ldots,m$, choose $\delta_j \in (\delta_{\min,j}, \rho_j)$ so that each segment $[z_{j-1},z_j]$ fits inside a disk $D(\tfrac{1}{2}(z_{j}+z_{j-1}),\delta_j)$ on which we can use Lemma~\ref{L:bound_on_disk} to find a bound $M_j$ on the integrand.
	\item For
	\[r_j=\operatorname{acosh}\frac{2\delta_j}{|z_j-z_{j-1}|}\]
	we have that the ellipse $L(z_{j-1},z_{j},r_j)$ is contained in the disk and hence that the integrand is bounded by $M_j$ on it.
	\item Use \eqref{E:Nbound} to compute $N_j$ so that $N_j$-th order Gauss-Legendre quadrature approximates $\int_{z_{j-1}}^{z_j}g(z)dz$ to within $\tfrac{1}{m} \Etol.$
    \item Add up the quadratures on each of the segments. 
\end{enumerate}
\end{strategy}
\begin{remark}
It makes sense in Step~(1) to continue bisecting as long as it reduces the cost function $N(g,\Etol)=\sum_{j=1}^m N_j$. The main gain here is that further bisection allows for larger values for $r_j$, and hence much less eccentric ellipses.
\end{remark}
\begin{remark}\label{R:strategy_beta_choice}
In practice we set $\beta\in (0,1)$ and continue splitting in Step~(1) until $\rho_j>\beta \delta_{\min,j}$ and set $\delta=\beta\rho_j$.
Experimentally, a value of $\beta = 0.912$ seems to work well, which has some asymptotic justification; see Remark~\ref{R:beta_choice}.
\end{remark}
\begin{remark}
For Step~(4), we find in the proof of Proposition~\ref{P:N1_asymptotics}, distributing the tolerance uniformly over the segments is asymptotically optimal. Alternatively, one can weight each tolerance by the length of the interval, so approximate to within $\tfrac{\abs{z_j - z_{j-1}}}{2}\Etol$.
This has the advantage that tolerance is independent on how other segments are split up. In practical situations, the difference in performance between the two approaches is minimal.
\end{remark}

The concept of splitting integration segments is not new, and indeed Neurohr provides one algorithm for splitting line segments in \cite[Algorithm 4.7.3]{NeurohrThesis}, which he concludes is a beneficial strategy. Neurohr's splitting strategy contains an inherent asymmetry, and avoiding this in Strategy~\ref{strat:main} will aid in computing the asymptotic behaviour in Proposition~\ref{P:N1_asymptotics}.

In order to assess how important the path splitting is in Strategy~\ref{strat:main}, we also consider an alternative naive strategy, where we fit the path in a single, possibly quite eccentric, ellipse. In order to bound the integrand, we cover the ellipse with discs on which we can use Lemma~\ref{L:bound_on_disk}.

\begin{strategy}\label{strat:ref} $ $\\
\begin{enumerate}
	\item Set $r=\beta\min\{ \operatorname{acosh}\big(\tfrac{1}{2}(|\alpha_i-1|+|\alpha_i+1|)\big): i = 1,\ldots,d_0 \} $ for some $\beta \in (0,1)$, to ensure that the critical points $\alpha_i$ are a positive distance away from $L(-1,1,r)$. As per Remark~\ref{R:strategy_beta_choice}, $\beta=0.912$ works well in practice.
	\item Determine a collection of discs $D(z_j,\delta_j)$ that cover $L(-1,1,r)$ and do not contain any of $\alpha_1,\ldots,\alpha_{d_0}$.
	\item Use Lemma~\ref{L:bound_on_disk} to find a bound $M_j$ on the integrand on each $D(z_j,\delta_j)$.
	\item Set $M=\max M_j$.
	\item Use \eqref{E:Nbound} to obtain $N$ so that $E_N(g)\leq\Etol$.
	\item Compute $I_N(g)$.
\end{enumerate}	
\end{strategy}

\begin{example}
	An example of the possible configurations achieved by the two strategies in the case of a function with pole at $z=0.3+0.4i$ is shown in Figure \ref{fig: example_method_configurations}. 
	\begin{figure}
		\centering
		\includegraphics[]{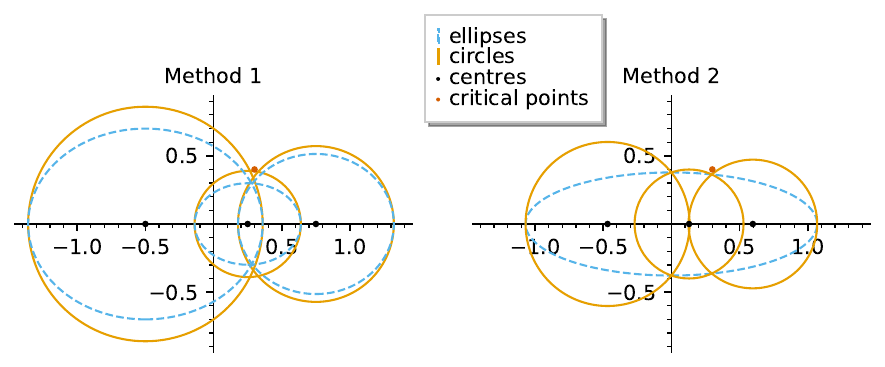}
		\caption{Example configuration of circles and ellipses for both strategies when there is a single critical point at $0.3+0.4i$.}
		\label{fig: example_method_configurations}
	\end{figure}
	Note the ellipses used in Strategy~\ref{strat:main} are much less eccentric than in Strategy~\ref{strat:ref}, corresponding to larger $r$.
\end{example}
The first strategy is similar to Petras' algorithm \cite{Petras2002}, and that used in Arb \cite{Johansson2013, Johansson2018}, though in the latter paper insufficient detail is given to recreate the method entirely.
Indeed, one could just use Arb, as was done in \cite{GaoThesis}. This works well enough to show that rigorous integration is practically viable, but discrepancies between desired use and interface prevent it from being truly competitive. For instance, variation in the integrand needs to be communicated to the library through rectangular regions rather than discs and the geometry of the critical points cannot be communicated to the library function either.

The second strategy can be considered as a rigorous version of the heuristic method \cite{Borwein2004},  where a single Gauss-Legendre quadrature is used for the entire line segment.

As we show in the next section, our knowledge about the behaviour of the integrand at critical points can in fact be used to obtain some remarkably sharp asymptotic results on the number of integration points required to obtain a given accuracy.

\section{Sensitivity to nearby critical points}\label{subsec: circle choosing}
As a test case, we analyse the performance of the two strategies for an integrand that has a pole nearby, given by
\[I_{z_0}=\int_{-1}^1 (z-z_0)^{-v}dz \quad\text{ for some }\quad v>0 \, .\]
We are particularly interested in $v<1$ because this is typical for differentials that lift to holomorphic forms on the corresponding Riemann surface. For instance, for $v=\frac{1}{2}$ we can take $w=(z-z_0)^2$ and get $(z-z_0)^{-\frac{1}{2}}dz=\frac{1}{2}dw$.

For given $\Etol$ and $v$, we consider the total number of function evaluations $N^{(1)}(z_0)$ for Strategy~\ref{strat:main} and $N^{(2)}(z_0)$ for Strategy~\ref{strat:ref} as a function of $z_0$.
\begin{figure}
    \centering
    \includegraphics[width=\textwidth]{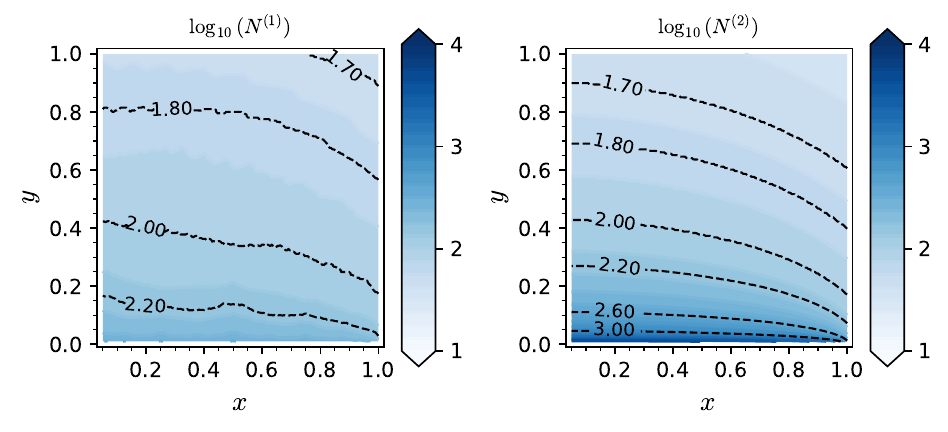}
    \caption{$\log(N^{(i)}(x+iy))$ for $x,y\in(0,1)$ for Strategies~\ref{strat:main} and \ref{strat:ref}, with $\Etol=2^{-100}$ and $v=\frac{1}{2}$.}
    \label{fig:proxy_heatmaps}
\end{figure}
    
In Figure~\ref{fig:proxy_heatmaps}, we plot $N^{(1)}(z_0)$ and $N^{(2)}(z_0)$ for $v=\frac{1}{2}$, $\Etol=2^{-100}$ with $z_0=x+iy$ for $x,y\in (0,1)$. A darker colour indicates a larger value. This already gives a good indication that, particularly for critical points close to the integration path (i.e. $y$ close to $0$), Strategy~\ref{strat:main} performs much better.
In the next two sections, we give asymptotic results confirming this.

\subsection{Asymptotics for a critical point approaching the integration path}

When we take $z_0 = iq$ for $q \in (0,1)$ we are able to compute the asymptotic behaviour for the two strategies as $q\to 0$. Note that this behaviour is not realistic for numerical computations, where one would probably ensure that $\Etol<q $. It does give us a qualitative explanation of the stark difference visible in Figure~\ref{fig:proxy_heatmaps}.

\begin{prop}\label{P:N1_asymptotics}
	For $q\to 0$ we have
	\[N^{(1)}(iq)\sim \begin{cases}
		\log(q^{-1})^2 & \text{ for }\;\; v\geq 1 \, , \\
		(\log(\log q^{-1}))^2 & \text{ for }\;\; 0 < v < 1 \, . \\		
	\end{cases}\]
\end{prop}

\begin{proof}
We split the integration path by iterative bisection.
We break up our integration path in $2K+1$ segments $[z_{j-1},z_j]$ in the following way.
\[
z_j=\begin{cases}
	-2^{-j} &\text{ for }j=0,\ldots,K \, , \\
	2^{-2K-1+j} &\text{ for }j=K+1,\ldots,2K+1 \, .\\
\end{cases}
\]
The centre for each of these segments is
\[
c_j=\begin{cases}
	-\frac{3}{2} 2^{-j} & \text{ for }j=1,\ldots,K \, , \\
	0 & \text{ for }j=K+1 \, , \\
	\frac{3}{2} 2^{j-2K-2} & \text{ for }j=K+2,\ldots,2K+1 \, . \\
\end{cases}
\]
We will choose $K$, $r$, and $\beta<1$ such that each segment fits inside an ellipse $L(z_{j-1},z_j,r)$ which in turn fits in a disc $D(c_j,\delta_j)$ such that the minimum distance from the disc to the point $iq$ is at least $(1-\beta)\delta_j$.
From fitting the ellipse inside the disk we obtain $|z_j-z_{j-1}|\cosh(r)\leq 2\delta_j$.

For $j=K+1$ with $c_j=0$ this means that \[2^{-K}\cosh(r)=\frac{1}{2}|z_{K+1}-z_K|\cosh{r}\leq\delta_{K+1}\leq \beta q,\]
and hence
\[K \geq -\log_2\beta +\log_2 q^{-1} + \log_2 \cosh(r).\]
	In particular, this allows us to choose $K$ on the order of $\log(q^{-1})$.

For each of the other segments, we can simply choose $\delta_j=\beta|c_j|$. In order to fit $L(z_{j-1},z_j,r)$ inside, we need $|c_j-z_j|\cosh(r)\leq \delta_j$, which means $\cosh(r)\leq 3\beta$. With $\beta<1$ this limits us to $r<1.7627\ldots$.

For $j\neq K+1$ we have that the distance of $iq$ to $D(c_j,\delta_j)$ is at least the distance to the origin, which is proportional to the length of the $j$-th integration interval. Therefore, independent of $q$, for $z\in D(c_j,\delta_j)$ we have
\[|(z-iq)^{-v}|\leq M_j= ((1-\beta)|c_j|)^{-v} = (3(1-\beta)|c_j-z_j|)^{-v} \, . \]

In order to approximate the full integral to within a tolerance $\Etol$, we use Gauss-Legendre quadrature on each interval with
\[N_j= \frac{1}{2r} \psquare{ \log\pround{\pi + \frac{64}{15\pround{e^{2r}-1}}} + \log M_j + \log\abs{c_j-z_j}- \log \gamma_j\Etol} \, , \]
where $\sum \gamma_j\leq 1$ with $0< \gamma_j\leq 1$. Concentrating on the part that depends on $q^{-1}$, we see that
\[N_j \sim (1-v)\log|c_j-z_j|-\log \gamma_j.\]
A Lagrange-multiplier argument shows that for the unrestricted minimisation of $\sum_j N_j$, the choice $\gamma_j = \frac{1}{2K+1}$ is optimal. With this choice, we have
\[N_j=N_{2K+2-j} \sim \log(2K+1)-(1-v)j\log 2 \;\text{ for }\; 1 \leq j \leq K.\]
Similarly, we find that
\[N_{K+1}\sim \log(2K+1)-(1-v)\log q^{-1}.\]
For $v\geq 1$, we see that $\sum N_j$ consists of $2K+1 \sim \log q^{-1}$ terms of size on the order of $\log q^{-1}$, so indeed we find
\[\sum N_j \sim (\log q^{-1})^2 \;\text{ for }\; v \geq 1.\]

For $0<v<1$ the asymptotic formula for $N_j$ can turn negative for $j$ close to $K+1$.
This reflects the fact that
 \[ \left \lvert \int_{z_{j-1}}^{z_{j}} (z-iq)^{-v}dz \right \lvert \leq q^{-v}|z_j-z_{j+1}|,\]
which for $j=K$ yields $q^{-v}2^{1-K}$. 
Hence, for $K$ such that
\begin{equation}\label{E:cutoff_K_inequality}
q^{-v}2^{1-K}\leq \frac{1}{2K+1}\Etol,
\end{equation}
we can leave out the central integral. Similarly, we have
\[ \left \lvert \int_{z_{j-1}}^{z_j}(z-iq)^{-v} \right \lvert \leq |z_j-z_{j-1}||z_j|^{-v}=2^{j(v-1)} \text{ for } j=1,\ldots, K,\]
so for $K\geq j \geq \frac{1}{1-v}(\log_2 (2K+1)-\log_2 \Etol)$ the approximation $0$ falls within the accuracy allotted to the segment. Using that our bounds are symmetric under $j\mapsto 2K+2-j$, we find
\[\sum_{j=1}^{\frac{1}{1-v}(\log_2 (2K+1)-\log_2 \Etol)}2 N_j\leq {\frac{2}{1-v}(\log_2 (2K+1)-\log_2 \Etol)} \log(2K+1)\sim \log(2K+1)^2.\] 
For $0<v<1$, the equation~\eqref{E:cutoff_K_inequality} is asymptotically satisfied by $K\sim \log q^{-1}$, so we find that for $v<1$, we can suffice with $\sum N_j \sim (\log\log q^{-1})^2$.
\end{proof}

\begin{prop}\label{P:N2_asymptotics}
	For $q\to 0$ we have $N^{(2)}(iq)\sim q^{-1}\log (q^{-1})$.
\end{prop}
\begin{proof}
We need to take $r=\beta \acosh(\sqrt{1+q^2})$ for some $\beta\in(0,1)$ to ensure that $L(-1,1,r)$ does not contain a critical point. Taking a series expansion gives $r \sim \beta q$. Our proxy bound then gives $M = \pround{q-\sinh(r)}^{-v} \sim ((1-\beta)q)^{-v}$, so we find
\begin{equation}\label{E:N-expansion simple case}
\begin{aligned}
    N^{(2)}(iq) &\sim \frac{1}{2\beta q} \psquare{\log \pround{\pi + \frac{32}{15 \beta q}} + v\log(q^{-1}) - v\log(1-\beta) - \log \Etol} \, ,  \\ 
    &\sim \frac{1}{2\beta} q^{-1}\psquare{(1+v)\log(q^{-1}) - \log(\beta(1-\beta)^v)-\log\Etol} \, .
\end{aligned}
\end{equation}
For $q\to 0$, this gives the result stated.
\end{proof}

\begin{remark}
    For a more numerically realistic example, one may want to consider $\Etol <q$, for instance by setting $\Etol = cq$ for some $c\in (0,1)$. Then one has that $N^{(1)}(iq) \sim \log(q^{-1})^2$ for any $v >0$, and that still $N^{(2)}(iq) \sim q^{-1} \log(q^{-1})$. 
\end{remark}

\begin{remark}\label{R:beta_choice}
Note that \eqref{E:N-expansion simple case} also shows how one could tailor the value of $\beta$ to a given error tolerance. Ignoring the values $q$ and $\Etol$ and setting, for instance, $v=\frac{1}{2}$, one would be left with minimising $-\frac{1}{\beta} \log(\beta \sqrt{1-\beta})$, which is accomplished with $\beta\approx 0.8$. As $\log(q^{-1}) - \log(\Etol) >0$, in practice a slightly larger value will be optimal. 
\end{remark}

\subsection{Integration paths passing between close-by critical points}
\label{S:close_pair_of_critical_points}

While the analysis in Section~\ref{subsec: circle choosing} illustrates that it is important to split paths near critical points, the example itself is less convincing. Apart from the fact that the integrand allows for an easy, exact antiderivative, one can generally change the integration path within its homology class, so one can generally avoid getting close to a single critical point.

As a more realistic test case, we consider the family of integrals
\[
I_q = \int_{-1}^1 \frac{dz}{\sqrt{4z^4-(16+4q^2+q^4)z^2-q^2(4+q^2)^2}} \, , \quad q \in (0,1) \, ,
\]
which has (among others) critical points at $\pm iq$, so as we let $q$ approach $0$, the integrand has critical points close together. Thus, if our application requires us to integrate along a path that goes in between the critical values, we cannot avoid getting close to them.

The differential form integrated in $I_q$ is a regular differential on the Riemann surface given by $w^2 - \big(z^4-(4+q^2+\tfrac{1}{4}q^4)z^2-q^2(2+\frac{1}{2}q^2)^2\big)=0$ of genus $1$, so in that sense it is the kind of differential one would integrate to compute a period matrix (although for Riemann surfaces of genus $1$ there are much better methods available). 

The integrand has finite critical points $z_0=\pm iq,\pm (2+\frac{1}{2}q^2)$ and expansion at $z=iq$ yields
\[\sqrt{\frac{i}{2q^5+24q^3+32q}}(z-iq)^{-\frac{1}{2}}+\cdots\,.\]
We expect that the growth of the integrand is dominated by the behaviour near the closest critical point, which will be $z=\pm iq$, so we consider as a replacement of Lemma~\ref{L:bound_on_disk} the bound
\[\Mproxy(z_0,\delta)=q^{-\frac{1}{2}}(|z_0-\alpha|-\delta)^{-\frac{1}{2}},\]
where $\alpha$ is a critical point of $g$ closest to $z_0$. 

\begin{figure}
	\centering
	\includegraphics[width=12cm]{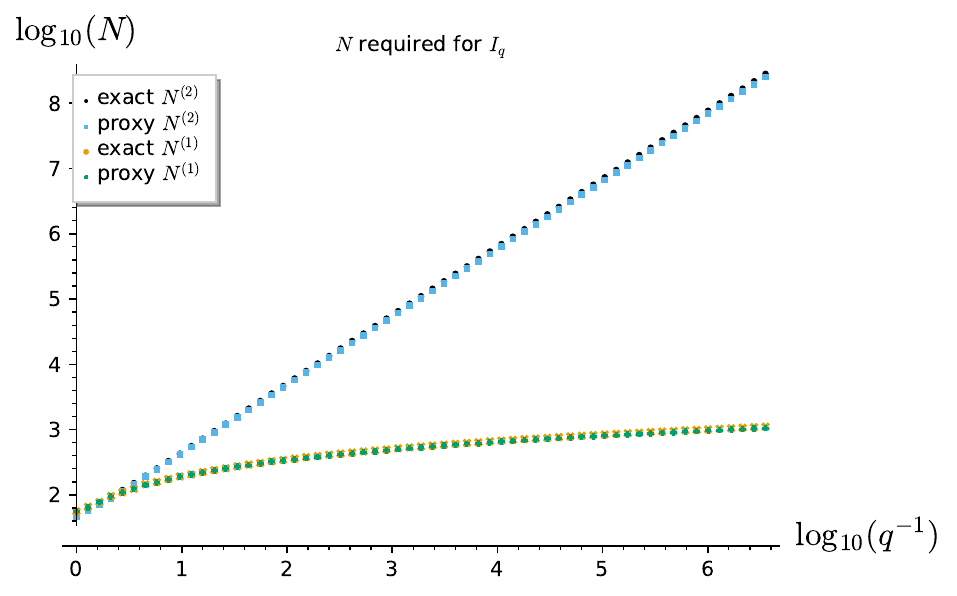}
	\caption{Proxy and actual bounds on number of nodes.}
\label{fig: N1_vs_N2_specific_curve}
\end{figure}
	
In Figure~\ref{fig: N1_vs_N2_specific_curve}, we plot the computed total $N$ for $\Etol=2^{-100}$ for Strategies~\ref{strat:main} and \ref{strat:ref}, both with $M(z_0,\delta)$ from Lemma~\ref{L:bound_on_disk} as well as $\Mproxy(z_0,\delta)$. We observe that
\begin{itemize}
	\item using $M(z_0,\delta)$ from Lemma~\ref{L:bound_on_disk} or $\Mproxy(z_0,\delta)$ hardly makes a difference, and 
	\item for Strategy~\ref{strat:ref} it looks like $N$ depends linearly on $q^{-1}$, whereas for Strategy~\ref{strat:main} it looks more like logarithmic.
\end{itemize}

We write $N^{(1)}(q)$ for the predicted value for $N$ using $\Mproxy$ when allowed splitting of the integration path versus $N^{(2)}(q)$ for when it is not. The following asymptotic result illustrates the importance of splitting paths.

\begin{prop}\label{P:N_asymptotics} For $q\to 0$, we have
$N^{(1)}(q)\sim (\log q^{-1})^2$ and $N^{(2)}(q)\sim q^{-1}\log q^{-1}$.
\end{prop}
\begin{proof}
For $N^{(1)}$ we end up with a similar argument as in Proposition~\ref{P:N1_asymptotics}. However, due to the dependence on $q^{-1}$ in $M$ we now find
\[N_j\sim \frac{1}{2}\log\abs{c_j-z_j}+\frac{{1}}{2}\log q^{-1}-\log\gamma_j.\]
Taking $\gamma_j = \frac{1}{2K+1}$ as before, this yields
\[N_j \sim \log(2K+1) - \frac{1}{2}j \log 2 + \frac{1}{2}K \log 2 = \log(2K+1) + \frac{1}{2}\pround{K-j}\log 2,\]
and from that we obtain $\sum_j N_j \sim K^2 \sim (\log q^{-1})^2$.  

For $N^{(2)}$, we get the same argument as in Proposition~\ref{P:N2_asymptotics}, but now our proxy bound gives $M = q^{-\frac{1}{2}}\pround{q-\sinh(r)}^{-\frac{1}{2}} \sim ((1-\beta)q^2)^{-\frac{1}{2}}$, so we find
\begin{equation*}
\begin{aligned}
    N^{(2)} &\sim \frac{1}{2\beta} q^{-1}\psquare{\log(q^{-1}) - \log(\beta\sqrt{1-\beta})-\log\Etol} \, .
\end{aligned}\qedhere
\end{equation*}
\end{proof}

\section{Comparison of performance on periods of smooth plane quartic curves}\label{sec: comparison with heuristic}

As evidence that the method proposed provides improved performance in practical situations, we consider the application in \cite{Bruin2019} to computing period matrices. We check that the new rigorous methods outperform the previously implemented heuristic method from \cite{Borwein2004}.

We generated Riemann surfaces determined by quartic plane curves determined by polynomials in $\mathbb{Z}[x,y]_{\leq 4}$ with coefficients uniformly randomly chosen from $\{-10,\ldots,10\}$.
We timed the computation of their Riemann matrices with $\Etol=2^{-100}$, i.e., to roughly $100$ binary digits.

We measured the runtimes $t_1,t_2$ for Strategies~\ref{strat:main} and \ref{strat:ref} respectively, as well as $t_h$ for the previously implemented heuristic method. As unit of time, we have taken the shortest measured runtime $t_{\min}$ for the heuristic method, in an effort to find a practical measure of runtime that is somewhat independent of actual hardware. In our case, we find $t_{\min}\approx 1.2$ seconds on a Intel Core i5-8350U CPU at 1.70GHz. In Figure \ref{fig: Ratios_of_runtimes}, we plot the ratios $t_1/t_{\min}$ (black) and $t_2/t_{\min}$ (yellow) against $t_h/t_{\min}$, with logarithmic axes to compress the extremes. The ratio $t_h/t_{\min}$ measures the difficulty of the computation, while the ratio $t_i/t_{\min}$ measures the performance of Strategy $i$, with a point laying underneath the diagonal if the corresponding strategy has outperformed the heuristic approach on this computation. 

\begin{figure}
    \centering
    \includegraphics[width=5in]{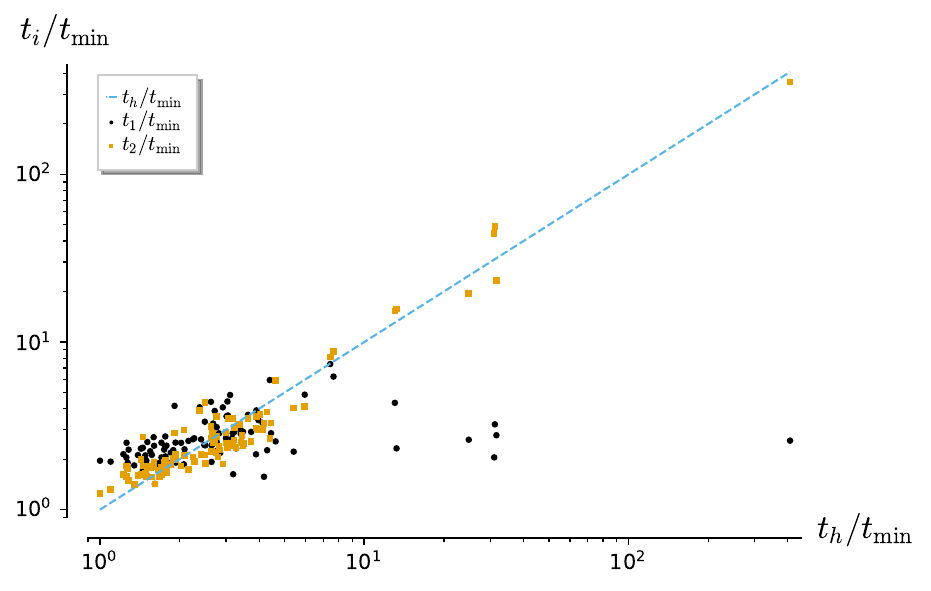}
    \caption{Ratio of runtimes of the rigorous methods ($t_i$) to the heuristic method ($t_h$). A reference time of $t_{\min} \approx 1.2$ seconds is taken.}
\label{fig: Ratios_of_runtimes}
\end{figure}
We see that generally, the curves that require only a short runtime under the heuristic method are relatively slower for the rigorous methods, as there is additional overhead for computing the $N$ required.

We see that our reference Strategy~\ref{strat:ref} generally performs comparably to the heuristic method, while our main Strategy~\ref{strat:main} performs reliably better on ``harder" examples. Where the transition occurs will depend on a variety of parameters, including the desired accuracy and the architecture of the hardware used. Experiments show that Strategy~\ref{strat:main} performs comparatively better for higher target accuracies (i.e. smaller values for $\Etol$) as well.

\section{Further Computational Considerations}

\subsection{Caching Nodes}\label{subsec: caching nodes}
Gauss-Legendre quadrature has notoriously hard to compute nodes and weights. A naive implementation using Bonnet's recursion formulas, that we refer to as REC, is asymptotically slow in that for very large orders $N$, it becomes the dominant part of computing Gauss-Legendre quadratures. Empirically, the time to complete a Gauss-Legendre quadrature $I_N$ using REC is approximately
\[\alpha (N+\gamma N^{1.7}),\]
with $\gamma$ approximately constant in the working precision, taking value 0.01.

In adaptive numerical integration approaches such as \cite{Borwein2004}, the order $N$ gets iteratively doubled until an error estimate based on the changes in approximations predicts a satisfactory accuracy. In those cases, nodes and weights for orders $N$ from only a very small set are used, so caching the relevant values is very worthwhile.

For the rigorous methods considered in the present document, where a required $N$ can be computed beforehand, there is a trade-off: one can round the $N$ used to promote reuse and cache nodes and weight, at the expense of computing some integrals to a higher order than required. We analysed the potential benefit of this in the following way.

Based on the integrals occurring in the experiments described in Section~\ref{sec: comparison with heuristic} we observed that the required $N$ to compute an integral to $D$ decimal digits accuracy is roughly shifted geometrically distributed with a minimum of about $D/25$ and a mean of about $D/2$. We validated that these estimates accurately reproduce the proportion of time spent on node computation. We assumed that 60 integrals would be computed in any given session, as this represented the average number of integrals required in the computation of the period matrix for Riemann surfaces in our sample. 

We used simulated annealing to find the best rounding strategies and determined that caching gave no advantage to this runtime metric for 100 decimal digits: node computation took up only a small proportion of the total runtime. Only at target accuracies of about 1000 decimal digits did the node computation take up to half the time.

Since we expect our code to be used mainly in scenarios where less than 1000 decimal digits are required, we opted to not implement any rounding of $N$ or caching. In other scenarios, different trade-offs may be appropriate, as we describe in the following section. Moreover, note that in our estimates we have been considering an integration routine that is vectorised, so caching strategies which cause an increasing in the number of function calls will automatically have worse performance compared to routines which do only one integration. In addition, for algebraic integrands an evaluation will likely involve a homotopy continuation method, which may be slower than for direct evaluation of a polynomial integrand for example. This will further lead to caching looking like a comparably less optimal strategy.

\subsection{Alternative Methods for Calculating Nodes}

There are other, asymptotically more efficient, algorithms to compute the nodes and weights for Gauss-Legendre integration. For instance, Neurohr \cite{NeurohrThesis} also considers the Glaser-Liu-Rohkling (GLR) algorithm for computing the nodes. This algorithm has $\mathcal{O}(N)$ complexity, as opposed to the $\mathcal{O}(N^2)$ complexity of REC, but due to overhead it is only more efficient for rather large values, as is tabulated in \cite[Table 3.2]{NeurohrThesis} by values of $N$ and working precision $D_{10}$ (in decimal digits).
For instance for $D_{10}=100, \, N=5000$, GLR can be an order of magnitude faster. These consideration, together with possible node caching strategies, should be taken into account for problems that require high accuracy.

\subsection{Alternative Integration Quadratures}
The proof of Theorem~\ref{T: N required} does not rely on Gauss-Legendre specifically and can be applied to any linear quadrature scheme for which the bound of \cite[Theorem 3.1.1]{NeurohrThesis} holds. Indeed in \cite{MolinNeurohr2019} the authors find that in the case of calculating the Abel-Jacobi map for hyperelliptic curves, the Gauss-Chebychev method is generically the best. There is thus scope for future work comparing the performance with these quadratures. 
\subsection{Rounding Error}\label{sec: rounding error}
When describing Strategies~\ref{strat:main} and \ref{strat:ref} as `rigorous', a necessary caveat must be made that (at least for the current implementation in Sage) floating point rounding in iterated summations may perturb the lower bits. Practically this is accounted for by having $\Etol$ a few orders of magnitude larger than the theoretical minimum tolerance for a given number of bits. Considering the specific case of Propositions~\ref{P:N1_asymptotics} and \ref{P:N_asymptotics}, we have approximately $\log q^{-1}$ summations, which over a practical implementation range would be less than $\log \Etol^{-1}$, and hence in this scenario it is always possible to choose a suitable number of guard digits as to ensure the rounding error is small compared to $\Etol$. To deal with these considerations completely one would likely wish to use ball arithmetic as per Arb. 
\subsection{Slow Convergence}
In our implementation, no evaluation limits were set in the event of a configuration where a large $N$ is required. This has the potential to be especially problematic in the case of Strategy~\ref{strat:main}, where the iterative splitting of the line segment could extend up to a point as to occupy a large portion of the available memory when storing the splitting. In our usage of the method, we never encountered such a situation, but in more general code one may wish to fail gracefully; see \cite{Johansson2018} for further discussions on this. 

\section*{Data Availability}
Strategy~\ref{strat:main} has been implemented in SageMath~9.6 \cite{SageMath} to integrate regular differentials on Riemann surfaces. It is used by default and can also be explicitly selected by setting \texttt{integration\_method = 'rigorous'} in the \texttt{RiemannSurface} class. The code, together with a test suite, is distributed in readable form as part of SageMath installations and can be downloaded from the main website \url{https://www.sagemath.org}.

\providecommand{\bysame}{\leavevmode\hbox to3em{\hrulefill}\thinspace}
\bibliographystyle{alpha}
\newcommand{\etalchar}[1]{$^{#1}$}

\end{document}